\documentclass[a4paper]{article}
\usepackage[english]{babel}
\usepackage[utf8x]{inputenc}
\usepackage[T1]{fontenc}
\usepackage[a4paper,top=3cm,bottom=2cm,left=2.7cm,right=2.7cm,marginparwidth=1.75cm]{geometry}
\usepackage{comment}
\usepackage{authblk}
\usepackage[misc]{ifsym}
\usepackage{indentfirst}
\usepackage[english]{babel}
\usepackage[utf8x]{inputenc}
\usepackage{amsmath,amsfonts,amssymb,amsthm}
\usepackage[english]{babel}
\usepackage{graphicx}
\usepackage{hyperref}
\usepackage{latexsym}
\usepackage{ifthen}
\usepackage{mathtools}
\usepackage{braket}
\usepackage{algorithm}
\usepackage{algpseudocode}
\usepackage{verbatim}
\usepackage[qm]{qcircuit}
\usepackage{todonotes}
\usepackage{multirow}
\usepackage[toc,page]{appendix}
\usepackage{listings}
\usepackage{xcolor}
\usepackage[nottoc,numbib]{tocbibind}
\usepackage{cprotect}
\usepackage{etex}

\makeatletter

\lstdefinelanguage{AMPL}{keywords={set,param,var,arc,integer,minimize,maximize,subject,to,node,sum,in,Current,complements,integer,solve_result_num,IN,contains,less,suffix,INOUT,default,logical,sum,Infinity,dimen,max,symbolic
,Initial,div,min,table,LOCAL,else,option,then,OUT,environ,setof ,union,all,exists,shell_exitcodeuntil,binary,forall,solve_exitcodewhile ,by,if,solve_messagewithin,check,in,solve_result
},sensitive=true,comment=[l]{\#}}

\theoremstyle{plain}
\newtheorem{thm}{Theorem}[section]
\newtheorem{lem}[thm]{Lemma}
\newtheorem{prop}[thm]{Proposition}
\newtheorem{cor}[thm]{Corollary}
\newtheorem{defi}[thm]{Definition}

\title{An adiabatic quantum algorithm for the Frobenius problem}
\author[1,2]{J. Ossorio-Castillo\thanks{Email: jqnssr@gmail.com (\Letter). Supported by MTM2016-75027-P (MEyC).}}
\author[2]{Jos\'e M. Tornero \thanks{Email: tornero@us.es. Supported by MTM2016-75027-P (MEyC) and P12-FQM-2696 (JdA).}}
\affil[1]{\small{\textit{Instituto Tecnol\'oxico de Matem\'atica Industrial (ITMATI), 15782 Santiago de Compostela, Spain}}}
\affil[2]{\textit{IMUS \& Departamento de \'Algebra, Universidad de Sevilla, 41012 Sevilla, Spain}}
\date{}

\begin{document}
\maketitle

\begin{abstract}
The (Diophantine) Frobenius problem is a well-known NP-hard problem (also called the stamp problem or the chicken nugget problem) whose origins lie in the realm of combinatorial number theory. In this paper we present an adiabatic quantum algorithm which solves it, using the so-called Apéry set of a numerical semigroup, via a translation into a QUBO problem. The algorithm has been specifically designed to run in a D-Wave 2X machine.
\end{abstract}

\section{Numerical semigroups and the Frobenius problem}

The study of numerical semigroups has its origins at the end of the 19th Century, when English mathematician James Joseph Sylvester (1814 -- 1897) and German mathematician Ferdinand Georg Frobenius (1849 -- 1917) were both interested in what is now known as the Frobenius problem, which we proceed to enunciate.

\begin{defi}\label{def:fcp}
Let $a_1, a_2,\ldots , a_n \in \mathbb{Z}_{\geq 0}$ with $\gcd (a_1, a_2,\ldots , a_n) = 1$, the Frobenius problem, or {\em FP}, is the problem of finding the largest positive integer that cannot be expressed as an integer conical combination of these numbers, i.e., as a sum

$$\sum_{j=1}^{n} \lambda_i a_i \text{ with } \lambda_i \in \mathbb{Z}_{\geq 0}.$$
\end{defi}

This problem, so easy to state, can be extremely complicated to solve in the majority of cases, as will be seen later. It can be found in a wide variety of contexts, being the most famous the problem of finding the largest amount of money which cannot be obtained with a certain set of coins: if, for example, we have an unlimited amount of coins of 2 and 5 units, we can represent any quantity except 1 and 3. \\

In order to understand the relationship between this problem and numerical semigroups, we shall first define the latter.

\begin{defi}
A semigroup is a pair $(S,+)$, where $S$ is a set and $+$ is a binary operation $+ : S \times S \rightarrow S$ that is associative.
\end{defi}

\begin{defi}
A numerical semigroup $S$ is a subset of the non-negative integers $\mathbb{Z}_{\geq 0}$ which is closed under addition, contains the identity element $0$, and has a finite complement in $\mathbb{Z}_{\geq 0}$.
\end{defi}

From now on, we shall denote numerical semigroups as $S$, taking for granted that they are commutative and that their associated operation is the addition. As it can be easily noted, a numerical semigroup is trivially a semigroup. In other words, a numerical semigroup is a semigroup that, additionally, is a monoid (i.e., it also has an identity element) and has finite complement in $\mathbb{Z}_{\geq 0}$. In order to work with numerical semigroups, it will be necessary to characterize them somehow. For that, let us set forth the following lemma. The proof of this result (and of any other result in this section, unless otherwise stated) can be found in \cite{rosales2009numerical}.

\begin{lem}
Let $A = \{a_1,...,a_n\}$ be a nonempty subset of $\mathbb{Z}_{\geq 0}$. Then, $$S = \langle A \rangle = \langle a_1,...,a_n \rangle = \{ \lambda_1a_1 + ... + \lambda_n a_n \ | \ \lambda_i \in \mathbb{Z}_{\geq 0}  \}$$ is a numerical semigroup if and only if $\gcd(a_1, ... , a_n) = 1$.
\end{lem}

The previous lemma tells us that, drawing from a set $A \subseteq \mathbb{Z}_{\geq 0}$, it is possible to generate a semigroup $S = \langle A \rangle$ as long as the elements of $A$ satisfy a certain condition. In this context, any set $A$ such that $S = \langle A \rangle$ for a certain numerical semigroup $S$ is called a system of generators of $S$. Even more, it can be proved that any numerical semigroup can be expressed that way, as will be shown next.

\begin{thm}
Every numerical semigroup $S$ admits a unique minimal system of generators, which can be calculated as $S^* \setminus (S^* + S^*)$ with $S^* = S \setminus \{0\}$.
\end{thm}

\begin{cor}
Let $S$ be a numerical semigroup generated by $A = \{ a_1 , \ldots , a_n \}$ with $ 0 \neq a_1 < a_2 < ... < a_n $. Then, $A$ is a minimal system of generators of $S$ if and only if $a_{i+1} \notin \langle a_1,a_2,\ldots,a_i \rangle$, for all $i \in \{1, \ldots , n-1 \}$.
\end{cor}

It will be stated later in this section that the minimal system of generators of a numerical semigroup is in fact finite. Hereinafter, if we say that $S = \langle A \rangle$ with $A = \{ a_1 , a_2 , \ldots , a_n \}$ is a numerical semigroup, then we shall assume without loss of generality that $a_1 < a_2 < \cdots < a_n$, $\gcd (a_1, a_2, \ldots, a_n) = 1$, and that $A$ is the minimal system of generators of $S$. Some examples of numerical semigroups, which will be used for the rest of the section, are:

\begin{equation*}
\begin{split}
\langle 3,7 \rangle = \{ 0,3,6,7,9,10,12,\rightarrow \} \\
\langle 4,9 \rangle = \{ 0,4,8,9,12,13,16,17,18,20,21,22,24,\rightarrow \} \\
\langle 5,8,11 \rangle = \{ 0,5,8,10,11,13,15,16,18,\rightarrow \} \\
\langle 5,7,9 \rangle = \{ 0,5,7,9,10,12,14,\rightarrow \}
\end{split}
\end{equation*} Where $\rightarrow$ means that all integers thenceforth are also included in the numerical semigroup. Having thus defined and characterized what numerical semigroups are, we proceed to describe some of their combinatorial invariants.

\begin{defi}
Let $S = \langle a_1 , a_2 , \ldots , a_n \rangle$ be a numerical semigroup, then $$m(S) = a_1 \text{ and } e(S) = n$$ are called respectively the multiplicity of $S$ and the embedding dimension of $S$.
\end{defi}

\begin{lem}
Let $S$ be a numerical semigroup, then
$$m(S) = \min S^*$$
\end{lem}

\begin{defi}
The set of gaps of a numerical semigroup $S$ is defined as $$G(S) = \mathbb{Z}_{\geq 0} \backslash S.$$ Its cardinal, $$g(S) = | G(S)|,$$ is called the genus of $S$; and its maximum, $$f(S) = \max G(S),$$ is called the Frobenius number of $S$.
\end{defi}

In other words, as by definition any numerical semigroup $S$ has a finite complement in $\mathbb{Z}_{\geq 0}$, we can define the maximum of such complement as $f(S)$, known as the Frobenius number. In fact, the Frobenius problem described at the beginning of this chapter in Definition \ref{def:fcp} can be enunciated as the problem of finding $f(S)$ for a certain numerical semigroup $S$. We shall expand the concepts related to the difficulties that surround the calculation of the Frobenius number later on. Table \ref{table:semiexs} shows the combinatorial invariants associated for the previously given examples of semigroups. \\

\begin{table}
\centering
\begin{tabular}{ |c|c|c|c|c|c| } 
 \hline
$S= \langle A \rangle$ & $m(S)$ & $e(S)$ & $G(S)$ & $g(S)$ & $f(S)$ \\ \hline
$\langle 3,7 \rangle$ & $3$ & $2$ & $\{ 1,2,4,5,8,11 \}$ & $6$ & $11$ \\ \hline
$\langle 4,9 \rangle$ & $4$ & $2$ & $\{ 1,2,3,5,6,7,10,11,14,15,19,23 \}$ & $12$ & $23$ \\ \hline
$\langle 5,8,11 \rangle$ & $5$ & $3$ & $\{ 1,2,3,4,6,7,9,12,14,17 \}$ & $10$ & $17$ \\ \hline
$\langle 5,7,9 \rangle$ & $5$ & $3$ & $\{1,2,3,4,6,8,11,13 \}$ & $8$ & $13$ \\ \hline 
\end{tabular}
\caption{Combinatorial invariants for some examples of semigroups}
\label{table:semiexs}
\end{table}

Greek-French mathematician Roger Apéry (1916 -- 1994), better known for proving in 1979 the irrationality of $\zeta(3)
$ \cite{apery1979irrationalite}, also laid the background in the context of the resolution of singularities of curves \cite{apery1946geometrie} for an important set associated to a numerical semigroup $S$ and one of its elements. \\

\begin{defi}
The Apéry set of a numerical semigroup $S$ with respect to a certain $s \in S^*$ is defined as $$Ap(S,s) = \{ x \in S \ | \ x-s \notin S \}.$$ 
\end{defi}

A possible characterization of the elements of the Apéry set one by one, which will come to special use in the last section, is given by the following lemma. \\

\begin{lem} \label{aperysetalt}
Let $S$ be a numerical semigroup and let $s \in S^*$. Then, $Ap(S,s) = \{ \omega_0, \omega_1, ..., \omega_{s-1} \}$ where $\omega_i$ is the least element of $S$ congruent with $i$ modulo $s$, for all $i \in \{0, ... , s-1 \}$. Consequently, $|Ap(S,s)| = s$.
\end{lem}

By means of an example on how to calculate the Apéry set of a semigroup with respect to a certain element, let $$S = \langle 5,8,11 \rangle = \{ 0,5,8,10,11,13,15,16,18,\rightarrow \}.$$ Then $Ap(S,5) = \{ \omega_0, \ldots , \omega_4 \}$, where

\begin{equation*}
\begin{split}
\omega_0 & = \min \{ x \in S \ | \ x \equiv 0 \mod 5 \} = 0 \\
\omega_1 & = \min \{ x \in S \ | \ x \equiv 1 \mod 5 \} = 11 \\
\omega_2 & = \min \{ x \in S \ | \ x \equiv 2 \mod 5 \} = 22 \\
\omega_3 & = \min \{ x \in S \ | \ x \equiv 3 \mod 5 \} = 8 \\
\omega_4 & = \min \{ x \in S \ | \ x \equiv 4 \mod 5 \} = 19 \\
\end{split}
\end{equation*} \\

We proceed to hint how the Apéry set proves some elementary (but nevertheless important) results of numerical semigroups.

\begin{prop}
The minimal system of generators of a numerical semigroup $S$ is finite.
\end{prop}

\begin{proof}
Let $s \in S^*$. Then, it is easy to see that for every $t \in S$ there exists a unique pair $(u,v) \in \mathbb{Z}_{\geq 0} \times Ap(S,s)$ such that $t = us + v$. Thus, $S$ is generated by $A = Ap(S,s) \ \cup \  \{ s \}$ and, as $A$ is finite, the unique minimal system of generators must be finite.
\end{proof}

\begin{lem}
Let $S$ be a numerical semigroup, then $$e(S) \leq m(S).$$
\end{lem}

\begin{proof}
Let $a = m(S)$ and let $A = Ap(S,a) \ \cup \  \{ a \}$. Thus, as $S$ can be generated by $A \setminus \{0\}$ and $|A \setminus \{0\}| = a$, we can conclude that $e(S) \leq m(S)$.
\end{proof}

The Apéry set is noteworthy in the context of the Frobenius problem as there exists a relationship between its members (regardless of the element $s \in S^*$ we choose) and the Frobenius number, which we proceed to enunciate and prove.

\begin{thm} \label{propfrobmaxapery}
{\em \textbf{(A. Brauer -- J. E. Shockley, 1962)} \cite{brauer1962problem}} Let $S$ be a numerical semigroup and let $s \in S^*$. Then
$$f(S) = \max Ap(S,s) - s$$
$$g(S) = \frac{1}{s} \left( \sum_{\omega \in Ap(S,s)}\omega \right) - \frac{s-1}{2}$$
\end{thm}

Now we exemplify the relationship between numerical semigroups and combinatorial optimization, as one of the most important problems in the latter branch of mathematics, known as the knapsack problem or rucksack problem, and more concretely one of its variants \cite{papadimitriou1998combinatorial} (p. 374), can be seen as the problem of deciding if a given integer $t$ belongs to a certain numerical semigroup $S$. 

\begin{defi}\label{def:ikp}
The numerical semigroup membership problem, or {\em NSMP}, is the problem of determining if, given a certain integer $t \in \mathbb{Z}_{\geq 0}$ and a numerical semigroup $S = \langle a_1,...,a_n \rangle$, the integer $t$ is contained in $S$. That is to say, if there exist non-negative integers $\lambda_1 , \ldots , \lambda_n \in \mathbb{Z}_{\geq 0}$ such that $$\sum_{i=1}^{n} \lambda_i a_i = t.$$
\end{defi}

The numerical semigroup membership problem is in {\em NP-complete}, as shown in \cite{papadimitriou1998combinatorial}. This fact was used by Jorge Ramírez-Alfonsín in 1996 \cite{ramirezalfonsin1996complexity} to finally prove that the Frobenius problem is in NP-hard (under Turing reductions). For that, he define a polynomial algorithm $\Lambda_{\text{NSMP}}$ for solving the NSMP which uses as a subroutine an unknown algorithm $\Lambda_{\text{FP}}$ that solves the Frobenius problem. Thus, he proved that the NSMP can be Turing reduced to the FP. As the NSMP is in NP-complete, he concluded that the FP is in NP-hard.\\

\section{The Ising spin problem in adiabatic quantum computing}

Our tackle of the Frobenius problem uses adiabatic quantum computing, more specifically that related to the Ising spin problem. This problem has a story where NP-hard problems are no strangers.\\

A form of adiabatic quantum computation \cite{mcgeoch2014adiabatic} is quantum annealing, where a known initial configuration of a quantum system evolves towards the ground state of a Hamiltonian that encodes the solution of an NP-hard optimization problem. The Canadian company D-Wave Systems announced in 2011 the first commercially available quantum annealer, composed of arrays of eight superconducting flux quantum bits with programmable spin–spin couplings, and published their results \cite{johnson2011quantum}. Subsequently in 2013, S. Boixo et al. \cite{boixo2013quantum} published their experimental results on the 108-qubit D-Wave One device. Their last chip, released in 2017 and called D-Wave 2000Q, has 2,048 qubits in a Chimera graph architecture \cite{dwave16}, and can be seen as a computer that solves the Ising spin problem, a particular type of quantum annealing which we proceed to describe. \\

The Ising spin model, originally formulated by physicist Wilhelm Lenz and first solved by his student, Ernst Ising \cite{ising1925report}, consists of a model of ferromagnetism in statistical mechanics in which we have to find the ground state of a system of $n$ interacting spins. If we represent the spins of these particles as binary variables $s_i \in \{-1,1\}$ with $i \in \{1,\ldots,n\}$, then the Ising Spin problem can be expressed as an integer optimization problem whose objective is to find the spin configuration of the system that minimizes the function $$H(s_1,\ldots,s_n) = \sum_{i=1}^{n} h_i s_i + \sum_{i=1}^{n-1} \sum_{j=i+1}^{n} J_{ij} s_i s_j,$$ where $h_i \in \mathbb{R}$ is the energy bias acting on particle $i$ (i.e. the external forces applied to each of the individual particles) and $J_{ij} \in \mathbb{R}$ is the coupling energy between the spins $i$ and $j$ (i.e. the interaction forces between adjacent spins). This problem was proved to be in NP-hard by Francisco Barahona \cite{barahona1982computational}, and can be effectively solved with the hardware implemented by D-Wave, whose chip permits to program independently the values of $h_i$ and $J_{ij}$ \cite{johnson2011quantum,boixo2013quantum}. It will be shown in the next section how to embed a certain optimization problem inside the D-Wave machine.

\section{An adiabatic quantum algorithm for the Frobenius problem}

The next algorithm shows the possibilities of calculating the Apéry set and the Frobenius number of a numerical semigroup with an actual adiabatic quantum computer. As described in the previous section, current D-Wave quantum annealers solve a certain kind of mathematical optimization problems known as Ising spin problems, namely $$H(s_1,\ldots,s_n) = \sum_{i=1}^{n} h_i s_i + \sum_{i=1}^{n-1} \sum_{j=i+1}^{n} J_{ij} s_i s_j$$ with $s_i \in \{-1,1\}$, where the objective is to find the minimum of $H(s_1,\ldots,s_n)$. \\

On the other hand, we recall that the Apéry set with respect to $s \in S \setminus \{0\}$, where $S = \langle a_1,\ldots,a_n \rangle$ is a numerical semigroup, is defined as $$Ap(S,s) = \{ x \in S \ | \ x-s \notin S \}.$$ The question is, how could we be able to translate this problem to the Ising spin model solved by the D-Wave machine? We have already shown in Lemma \ref{aperysetalt} that $Ap(S,s) = \{\omega_0,\ldots,\omega_{n-1}\}$, where $$\omega_i = \min \{x \in S : x \equiv i \mod s \}.$$ How about transforming the definition of $\omega_i$ into the answer of a mathematical optimization problem? \\

\begin{defi}
{\em \cite{papadimitriou1998combinatorial}} An integer linear program, or {\em ILP}, is defined in its canonical form as the optimization problem:

\begin{center}
	\begin{tabular}{rl}
		$\min$ & $c^Tx$\\
		subject to: & $Ax = b$\\
        &
	\end{tabular}
\end{center} where $x \in \mathbb{Z}_{\geq 0}^n $, $A \in \mathbb{Z}^n \times \mathbb{Z}^m$, $b \in \mathbb{Z}^m$ and $c \in \mathbb{Z}^n$.
\end{defi}

Thus, it is straightforward to see that the calculation of each of the $\omega_i$ can be redefined as the following ILP:

\begin{equation*}
\begin{array}{rl}
\text{min}  & \displaystyle\sum\limits_{j=1}^{n} a_{j}x_{j}\\
\textit{subject to:} & \displaystyle\sum\limits_{j=1}^{n}   a_{j}x_{j} = i + sk,\\
\end{array}
\end{equation*} with $x_{j} \in \mathbb{Z}_{\geq 0}$ for all $j=1 ,..., n$ and $k \in \mathbb{Z}_{\geq 0}$. \\ 

This mathematical optimization problem represents a way of calculating the Apéry set in its own right and, although integer linear programming is in NP-hard and its recognition version (deciding whether $Ax = b$ has a feasible solution or not regardless of its optimality) is in NP-complete \cite{papadimitriou1998combinatorial}, this computational complexity arises in the general case. In our context, it may prove to be an easier problem; however, details on this remain to be worked out. This approach for the calculation of the Apéry set first appeared in \cite{greenberg1980algorithm}, although Greenberg's work dealt with the direct calculation of the Frobenius number by means of $Ap(S,a_1)$, as will be discussed later. \\

We have tested this algorithm using state-of-the-art optimization software for integer optimization. In our case, the optimization problem for $\omega_i$ was modeled using AMPL \cite{fourer1990modeling,fourer2002ampl}, an algebraic modeling language for solving mathematical optimization problems. AMPL has a clear advantage, as our problem is entirely parameterized and AMPL allows us to describe the generic problem in a \verb|.mod| file while defining the actual values for the parameters in a separate \verb|.dat| file. This way, we have just to change the parameters inside the \verb|.dat| file in order to solve a new instance of the Apéry set. The \verb|.mod| file we propose is shown in Figure \ref{fig::aperysetmodfile} (the names of the parameters and variables of the problem are maintained so that no further explanation is required). \\

\begin{figure}[!b]
\begin{lstlisting}[language=AMPL]
	param n;
	set N := 1..n;
	param a {N};
	param s;
	param i;
	var X {N} integer, >= 0;
	var K integer;
	
	minimize T:
    	sum {j in N} a[j]*X[j];

	subject to C:
		sum {j in N} a[j]*X[j] = i + s*K;
\end{lstlisting}
\cprotect\caption{File \verb|apery_set_member.mod|}
\label{fig::aperysetmodfile}
\end{figure}

Let us suppose that we are interested in the numerical semigroup $S = \langle 11, 19, 23 \rangle$, and that we want to calculate the Apéry set of $s = 30$ (i.e., $Ap(S,30)$). The calculation of, for example, $\omega_5 \in Ap(S,30)$, will have the associated \verb|.dat| file depicted in Figure \ref{fig::aperysetdatfile}. In order to compute $\omega_5$, we also need a \verb|.run| file that will load the model and the data of the problem, and which will also call the solver for solving the problem. In our case, the solver we have chosen is Gurobi \cite{gurobi} which, among other things, solves integer linear problems. The \verb|.run| file we have used is shown in Figure \ref{fig::aperysetrunfile}, and the corresponding output we obtain can be seen in Figure \ref{fig::aperysetrunoutput}. This output tells us that $\omega_5 = 65$ and also that a representation of 65 with respect to the generators of the semigroup is $$65 = 0 \times (11) + 1 \times (19) + 2 \times (23).$$\\

\begin{figure}[!b]
\begin{lstlisting}[language=AMPL]
	param n := 3;
	param a :=
	    1   11
	    2   19
	    3   23;
\end{lstlisting}
\cprotect\caption{File \verb|numerical_semigroup.dat|}
\label{fig::aperysetdatfile}
\end{figure}

\begin{figure}[!b]
\begin{lstlisting}[language=AMPL]
	model apery_set_member.mod;
	data numerical_semigroup.dat;
	let s := 30;
	let i := 5;
	option solver gurobi;
	solve;
	display X;
	display K;
\end{lstlisting}
\cprotect\caption{File \verb|apery_set_member.run|}
\label{fig::aperysetrunfile}
\end{figure}

\begin{figure}[!t]
\begin{lstlisting}
	Gurobi 8.0.0: optimal solution; objective 65
	3 simplex iterations
	1 branch-and-cut nodes
	X [*] :=
	1  0
	2  1
	3  2
	;

	K = 2
\end{lstlisting}
\cprotect\caption{Output for \verb|apery_set_member.run|}
\label{fig::aperysetrunoutput}
\end{figure}

We can also write an alternative \verb|.run| file that will directly calculate and display the whole content of the Apéry set for a certain $s \in S$. First, we have to drop the \verb|param i := 5;| line from the \verb|.run| file, as it will be changed in each iteration of the main loop, and modify the \verb|.run| so that it will look like the one in Figure \ref{fig::aperysetcompleterun}. Thus, we obtain $Ap(S,30) = \{0,11,19,22,23,33,34,38,42,44,45,46,55,56,$ $57,61,65,66,67,69,
77,78,80,84,88,89,92,
100,103,111\}$ (see Figure \ref{fig::aperyset30} for the actual output).\\

\begin{figure}[!t]
\begin{lstlisting}[language=AMPL]
	model apery_set_member.mod;
	data numerical_semigroup.dat;
	let s := 30;
	option solver gurobi;
	param apery_set {0..s-1};
	for {l in 0..s-1} {
	    let i := l;
	    solve;
	    let apery_set[l] := T;
	}
	display apery_set;
\end{lstlisting}
\cprotect\caption{File \verb|apery_set.run|}
\label{fig::aperysetcompleterun}
\end{figure}

\begin{figure}[!t]
\begin{lstlisting}
	apery_set [*] :=
	 0   0    4  34    8  38   12  42   16  46   20  80   24  84   28  88
	 1  61    5  65    9  69   13 103   17  77   21 111   25  55   29  89
	 2  92    6  66   10 100   14  44   18  78   22  22   26  56
	 3  33    7  67   11  11   15  45   19  19   23  23   27  57
	;
\end{lstlisting}
\cprotect\caption{Output for \verb|apery_set.run|}
\label{fig::aperyset30}
\end{figure}

This algorithm also provides a way for calculating the Frobenius number of a numerical semigroup. We recall that, for any numerical semigroup $S$ and any integer $s \in S \setminus \{0\}$, then $$f(S) = \max \{Ap(S,s)\}-s.$$ 

Thus, as the difficulty of obtaining $f(S)$ this way increments with respect to the number $s$ we choose (we have to solve $s$ ILPs), the smartest way of proceeding is by solving it in the case where $s = \min (S \setminus \{0\})$ (i.e., $s = a_1$, as done by \cite{greenberg1980algorithm}). The AMPL file that represents this approach is shown in Figure \ref{fig::frobeniusnumberrun}. In our case, the output (Figure \ref{fig::frobeniusoutput}) tells us that $f(S) = 81$. All these files can be found in the public GitHub repository \cite{ossoriocastillo2018numsem}; however, a license for both AMPL and Gurobi is needed in order to run them and obtain the same results (or any result at all). \\

\begin{figure}[!b]
\begin{lstlisting}
	model apery_set_member.mod;
	data numerical_semigroup.dat;
	let s := a[1];
	option solver gurobi;
	param m default 0;
	for {l in 0..s-1} {
	    let i := l;
	    solve;
	    if T > m then let m := T;
	}
	param f := m - s;
	display f;
\end{lstlisting}
\cprotect\caption{File \verb|frobenius_number.run|}
\label{fig::frobeniusnumberrun}
\end{figure}

\begin{figure}[!b]
\begin{lstlisting}
	f = 81
\end{lstlisting}
\cprotect\caption{Output for \verb|frobenius_number.run|}
\label{fig::frobeniusoutput}
\end{figure}

What we have shown is a classical algorithm for obtaining the Apéry set and the Frobenius number in a general manner that depends on a black box that solves an ILP with global optimality. From now on we will explain the steps we have followed in order to solve this ILP with an adiabatic quantum computer and, specifically, with a D-Wave 2X machine (and also the obstacles we have encountered in that path). In order to transform this ILP problem into the Ising model solved by the D-Wave hardware, one step further involves changing its integer variables into a new set of binary variables with at most a polynomial cost.

\begin{defi}
{\em \cite{papadimitriou1998combinatorial}} A binary linear program, or {\em 0-1LP}, is defined in its canonical form as:

\begin{center}
	\begin{tabular}{rl}
		$\min$ & $c^Tx$\\
		subject to: & $Ax = b$\\
        &
	\end{tabular}
\end{center} where $x \in \{0,1\}^n$, $A \in \mathbb{Z}^n \times \mathbb{Z}^m$, $b \in \mathbb{Z}^m$ and $c \in \mathbb{Z}^n$.
\end{defi}

Both problems are polynomially equivalent, as shown in \cite{papadimitriou1998combinatorial} (Theorem 13.6), where an upper bound for the number of binary variables representing each integer variable from the original ILP problem is given. However, we can tight this number of binary variables by using our knowledge of the problem and one of the lower bounds of the Frobenius number given in \cite{ramirezalfonsin2005diophantine} (ideally, we could use $f(S)$). In 1935, Russian mathematician Issai Schur proved in a lecture in Berlin \cite{brauer1942problem,ramirezalfonsin2005diophantine} the following result: \\

\begin{thm}
{\em \textbf{(I. Schur, 1935)}} Let $S = \langle a_1, ..., a_n \rangle$ be a numerical semigroup. Then,
\begin{equation*}
f(S) \leq (a_1-1)(a_n-1)-1.
\end{equation*}
\end{thm} 

Thus, if we define $$t_j = 1 + \left\lfloor \log_2\left(\dfrac{(a_1-1)(a_n-1) + s - 1}{a_j}\right) \right\rfloor $$ for every $j \in \{1,\ldots,n\}$, and $$u = 1 + \left\lfloor \log_2\left(\dfrac{(a_1-1)(a_n-1) + s - i - 1}{s}\right) \right\rfloor$$ we have given an upper bound for the number of bits needed to describe every variable $x_j$ in the worst case (we can also run every known bound for the Frobenius number, and stick to the minimum of them, but for now let us just use Schur's result). It follows that the transformation of our problem from ILP form to 0-1LP delivers the minimization problem shown below:

\begin{equation*}
\begin{array}{rll}
\text{min}  & \displaystyle\sum\limits_{j=1}^{n} a_{j}\displaystyle\sum\limits_{l=0}^{t_j}2^l x_{jl}\\
\textit{subject to:} & \displaystyle\sum\limits_{j=1}^{n}   a_{j}\displaystyle\sum\limits_{l=0}^{t_j}2^lx_{jl} = i + s\sum\limits_{m=0}^{u}2^m k_{m}
\end{array}
\end{equation*} with $x_{jl} \in \{0,1\}$ and $ k_m \in \{0,1\}$ for all $j$, $l$ and $m$ contained in the summation indices. \\

Let us come back to the definition of the Ising model. In it, we have binary variables which can have the values -1 or 1. On the other hand, in our 0-1LP for calculating the elements of the Apéry set the variables can be in 0 or 1. It is usually more advantageous to redefine the Ising spin problem as a quadratic unconstrained binary optimizacion problem, or QUBO, which consists of the minimization of the objective function
$$Q(x_1,\ldots,x_n) = c_0 + \sum_{i=1}^{n} c_i x_i + \sum_{1 \leq i < j \leq n} q_{ij} x_i x_j,$$ with $x_i \in \{0,1\}$, thus changing the possible states of its variables from $1$ and $-1$ to $1$ and $0$. The transformation follows from $s_i = 1 - 2x_i$, and can be easily checked. In fact, the documentation from D-Wave allows to program our problem directly into a QUBO formulation \cite{dwave16a}, besides the Ising model formulation. One detail remains to be figured out: QUBO problems have no constraints, but ours have one. How to proceed? \\

The most common way to transform a constrained optimization problem into an unconstrained one is by penalizing the constraints, putting them in the objective function. In this manner, we can force the fulfillment of the constraints of the problem by punishing the error committed in them. In our example, as we have a single equality constraint, namely $$\displaystyle\sum\limits_{j=1}^{n} a_{j}x_{j} = i + sk,$$ we can force it to the objective function this way:

\begin{equation*}
\begin{array}{rl}
\text{min}  & \displaystyle\sum\limits_{j=1}^{n} a_{j}x_{j} + \lambda^{(\nu)} \left( \displaystyle\sum\limits_{j=1}^{n}   a_{j}x_{j} - i - sk \right)^2\\
\end{array}
\end{equation*} where $\lambda^{(\nu)}$ is updated after every iteration in the following way, until no change in $\lambda$ is obtained:

\begin{equation*}
\lambda^{(\nu + 1)} = \lambda^{(\nu)} + \alpha \cdot \left| \displaystyle\sum\limits_{j=1}^{n}   a_{j}x_{j} - i - sk \right|
\end{equation*} with $\alpha > 0$. \\

In 0-1LP form, we obtain the following QUBO problem:

\begin{equation*}
\begin{array}{rll}
\text{min}  & \displaystyle\sum\limits_{j=1}^{n} a_{j}\displaystyle\sum\limits_{l=0}^{t_j}2^l x_{jl} + \lambda^{(\nu)} \left( \displaystyle\sum\limits_{j=1}^{n}   a_{j}\displaystyle\sum\limits_{l=0}^{t_j}2^lx_{jl} - i - s\sum\limits_{m=0}^{u}2^m k_{m} \right)^2,
\end{array}
\end{equation*} with its corresponding penalty updating formula:

\begin{equation*}
\lambda^{(\nu + 1)} = \lambda^{(\nu)} + \alpha \cdot \left| \displaystyle\sum\limits_{j=1}^{n}   a_{j}\displaystyle\sum\limits_{l=0}^{t_j}2^lx_{jl} - i - s\sum\limits_{m=0}^{u}2^m k_{m} \right|
\end{equation*} with $\alpha > 0$. \\

Now, we are going to test the performance of this new procedure for obtaining the Apéry set (and the Frobenius number) via a classical solver. For that, we have just to modify the original \verb|apery_set_member.mod| file into a new one with an unconstrained ILP, with the constraint penalized in the objective function as previously exaplained. This new AMPL file can be found in Figure \ref{fig::aperysetmemberlagrmod}, along with its corresponding \verb|.run| file in Figure \ref{fig::aperysetlagrrun}. Additionaly, we are going to keep track of the number of iterations of the algorithm (i.e., $\nu$), and on the values of $\lambda^{(\nu)}$ needed for the algorithm to stop. This way, if we start with $\lambda^{(0)} = 0$, the output for this algorithm is found in Figure \ref{fig::outputaperysetlagrrun}. On the other hand, if we start with $\lambda^{(0)} = 100$, the output is shown in Figure \ref{fig::outputaperysetlagrrun_b}. \\

\begin{figure}[!t]
\begin{lstlisting}
	param n;
	set N := 1..n;
	param a {N};
	param s; #
	param i; #
	param lambda;
	param rest;
	var X {N} integer, >= 0;
	var K integer;

	minimize T:
    	sum {j in N} a[j]*X[j] 
    	+ lambda*((sum {j in N} a[j]*X[j] - i - s*K)^2);
\end{lstlisting}
\cprotect\caption{File \verb|apery_set_member_lagr.mod|}
\label{fig::aperysetmemberlagrmod}
\end{figure}

\begin{figure}[!b]
\begin{lstlisting}
	model apery_set_member_lagr.mod;
	data numerical_semigroup.dat;
	let s := 30;
	param apery_set {0..s-1};
	param lambdas {0..s-1};
	param iterations {0..s-1};
	param it default 0;
	option solver gurobi;
	
	for {l in 0..s-1} {
	    let i := l;
	    let lambda := 0;
	    let it := 0;
	    let rest := 1;
	    repeat {
	        solve;
	        let rest := abs(sum {j in N} a[j]*X[j] - i - s*K);
	        let lambda := lambda + rest;
    	    let it := it + 1;
    	} while rest != 0;
    	let apery_set[l] := T;
    	let lambdas[l] := lambda;
    	let iterations[l] := it;
	}

	display apery_set;
	display iterations;
	display lambdas;
\end{lstlisting}
\cprotect\caption{File \verb|apery_set_lagr.run|}
\label{fig::aperysetlagrrun}
\end{figure}

\begin{figure}[!t]
\begin{lstlisting}
	apery_set [*] :=
 	0   0    4  34    8  38   12  42   16  46   20  80   24  84   28  88
 	1  61    5  65    9  69   13 103   17  77   21 111   25  55   29  89
 	2  92    6  66   10 100   14  44   18  78   22  22   26  56
 	3  33    7  67   11  11   15  45   19  19   23  23   27  57
	;

	iterations [*] :=
0  1    3  2    6  2    9  2   12 22   15  2   18  2   21 11   24 40   27  2
1 63    4  2    7  2   10 22   13 21   16  2   19  2   22  2   25  2   28  2
2  2    5  2    8  2   11  2   14  2   17  2   20 44   23  2   26  2   29  3
;

lambdas [*] :=
 	0   0    4  34    8  68   12  32   16  46   20  62   24  62   28  58
 	1  62    5  35    9  39   13  62   17  47   21  90   25  85   29  90
 	2  62    6  66   10  90   14 104   18  78   22 112   26  56
 	3  93    7  67   11 101   15  45   19  79   23  23   27  57
	;
\end{lstlisting}
\cprotect\caption{Output for  \verb|apery_set_lagr.run| and $\lambda^{(0)} = 0$}
\label{fig::outputaperysetlagrrun}
\end{figure}

\begin{figure}[!t]
\begin{lstlisting}
apery_set [*] :=
 0   0    4  34    8  38   12  42   16  46   20  80   24  84   28  88
 1  61    5  65    9  69   13 103   17  77   21 111   25  55   29  89
 2  92    6  66   10 100   14  44   18  78   22  22   26  56
 3  33    7  67   11  11   15  45   19  19   23  23   27  57
;

iterations [*] :=
 0 1    3 1    6 1    9 1   12 1   15 1   18 1   21 1   24 1   27 1
 1 1    4 1    7 1   10 1   13 1   16 1   19 1   22 1   25 1   28 1
 2 1    5 1    8 1   11 1   14 1   17 1   20 1   23 1   26 1   29 1
;

lambdas [*] :=
 0 100    4 100    8 100   12 100   16 100   20 100   24 100   28 100
 1 100    5 100    9 100   13 100   17 100   21 100   25 100   29 100
 2 100    6 100   10 100   14 100   18 100   22 100   26 100
 3 100    7 100   11 100   15 100   19 100   23 100   27 100
;
\end{lstlisting}
\cprotect\caption{Output for  \verb|apery_set_lagr.run| and $\lambda^{(0)} = 100$}
\label{fig::outputaperysetlagrrun_b}
\end{figure}

Foreseeably, the Apéry set we obtain for $S$ and $s = 30$ is correct, but that should not be a surprise as both optimization problems are equivalent to the original one. What is of interest here is that a low starting point for $\lambda$ implies a relatively large number of iterations, which translates in more calls to our integer linear programming solver than in the original constrained version. However, if we start with a large enough value of $\lambda$, we can obtain the same results without additional calls to the solver and just one iteration for each of the elements of the Apéry set. On the other hand, if we exceed in the value for $\lambda$, we would also need just one iteration per $\omega_i$, but the solver would need more time in order to find the global optimum. The latter behavior can be observed when taking into account the number of simplex iterations and branch-and-cut nodes that Gurobi needs for each of the problems. \\

Nevertheless, this unconstrained reformulation of the problem is not meant to be solved by Gurobi or any other integer linear programming solver, but rather by an ideal adiabatic quantum computer, and more concretely by one of the few D-Wave machines that currently exist worldwide. Our unconstrained problem, in QUBO form, is shown below (please note that, if $x \in \{0,1\}$, then $x = x^2$):

$$c_0 
+ \sum_{m=0}^{u} c_m k_m +
\sum_{j=1}^{n} \left( \sum_{l=0}^{t_j} c_{jl} x_{jl} \right) 
+ \sum_{1 \leq j < j' \leq n} \left( \sum_{0 \leq l < l' \leq \min(t_j,t_{j'})}i q_{jl,j'l'} x_{jl} x_{j'l'} \right)$$

$$+ \sum_{j=1}^{n} \left[ \sum_{l=0}^{t_j} \left( \sum_{m=0}^{u} q_{jl,m} x_{jl} k_m \right) \right] + \sum_{0 \leq m < m' \leq u} q_{m,m'}, $$ where 

\begin{equation*}
\begin{array}{rll}
c_0 & = \lambda^{(\nu)} i^2\\
c_m & = \lambda^{(\nu)} \left( 2^{2m} s^2 + 2^{m+1} is \right)\\
c_{jl} & = 2^l a_j + \lambda^{(\nu)} \left( 2^{2l} a_j^2 - 2^{l+1} a_j i \right)\\
q_{m,m'} & = 2^{m+m'+1} \lambda^{(\nu)} s^2 \\
q_{jl,m} & = 2^{l+m+1} \lambda^{(\nu)} a_j s \\
q_{jl,j'l'} & = 2^{l+l'+1} \lambda^{(\nu)} a_j a_{j'}.
\end{array}
\end{equation*}

Right now, an ideal quantum annealer should be able to solve our QUBO problem for finding the members of the Apéry set, provided that such a computer has a sufficiently large enough amount of qubits, and also that the graph connecting the qubits is complete. However, to date there are no quantum annealers that fulfill those requirements (they may be available in the future, though). The latest quantum annealers commercially available are the D-Wave 2X, with 1152 and 3360 couplers (connections between adjacent qubits), and the D-Wave 2000Q, which has 2048 qubits and 6016 couplers. Their graphs are far from being complete, as every qubit in their grids are at most connected with six other qubits, as shown in Figure \ref{fig::dwavequbo} (it may be less than that, as some qubits may be off after the last recalibration of the machine). \\

The importance of the completeness of the graph is problem dependant. Our QUBO instance, for example, has a complete connectivity graph between its variables. With an ideal quantum annealer we would have no problem, but with the D-Wave machine it is mandatory to transform our problem graph into an alternative graph that could be embedded into the Chimera graph. In other words, we have to solve an instance of the subgraph isomorphism problem, which happens to be in NP-complete. Even more, our problem may not be embeddable inside the D-Wave (for example, for the 1152 qubit Chimera grid of the D-Wave 2X, the largest complete graph that can be embedded into it is believed to be the $K_{33}$). Research on the subject of embedding a problem graph into D-Wave's Chimera graph can be found in \cite{choi2008minor} and \cite{choi2011minor}, where the concepts of \textit{embedding} and \textit{parameter setting} are explained. \\

\begin{figure}[!b]
\begin{lstlisting}
 p qubo 0 3 3 2
 0 0 2.6
 1 1 4.5
 2 2 -1.8
 0 1 3.5
 1 2 2.0
\end{lstlisting}
\cprotect\caption{Example of \verb|.qubo| file}
\label{fig::qubofileexample}
\end{figure}

\begin{figure}[!b]
\begin{lstlisting}
 3 bits,  find Min, SubMatrix= 47, -a o, timeout=2592000.0 sec
 001
 -1.80000 Energy of solution
 0 Number of Partitioned calls, 1 output sample
 0.03002 seconds of classic cpu time
\end{lstlisting}
\cprotect\caption{Output for the example \verb|.qubo| file}
\label{fig::qbsolvoutputexample}
\end{figure}

There is a way to skip these limitations, by solving subinstances of our graph instead of the complete graph. For that, D-Wave released a graph partitioning open source library called \verb|qbsolv| \cite{qbsolv}. Its corresponding executable needs a certain kind of file format (called \verb|.qubo|), to work. For example, if our QUBO instance is defined by $$2.6 x_0 + 4.5 x_1 - 1.8 x_2 + 3.5 x_0 x_1 + 2 x_1 x_2,$$ where $x_0,x_1,x_2 \in \{0,1\},$ its corresponding \verb|.qubo| file is the one shown in Figure \ref{fig::qubofileexample}. The format is quite: the first line always starts with \verb|p qubo 0|, followed by the number of variables, the number of nonzero diagonals, and the number of nonzero couplings. The output obtained by \verb|qbsolv| for this file can be seen in Figure \ref{fig::qbsolvoutputexample}. \\

\begin{figure}[!t]
\includegraphics[height=12cm]{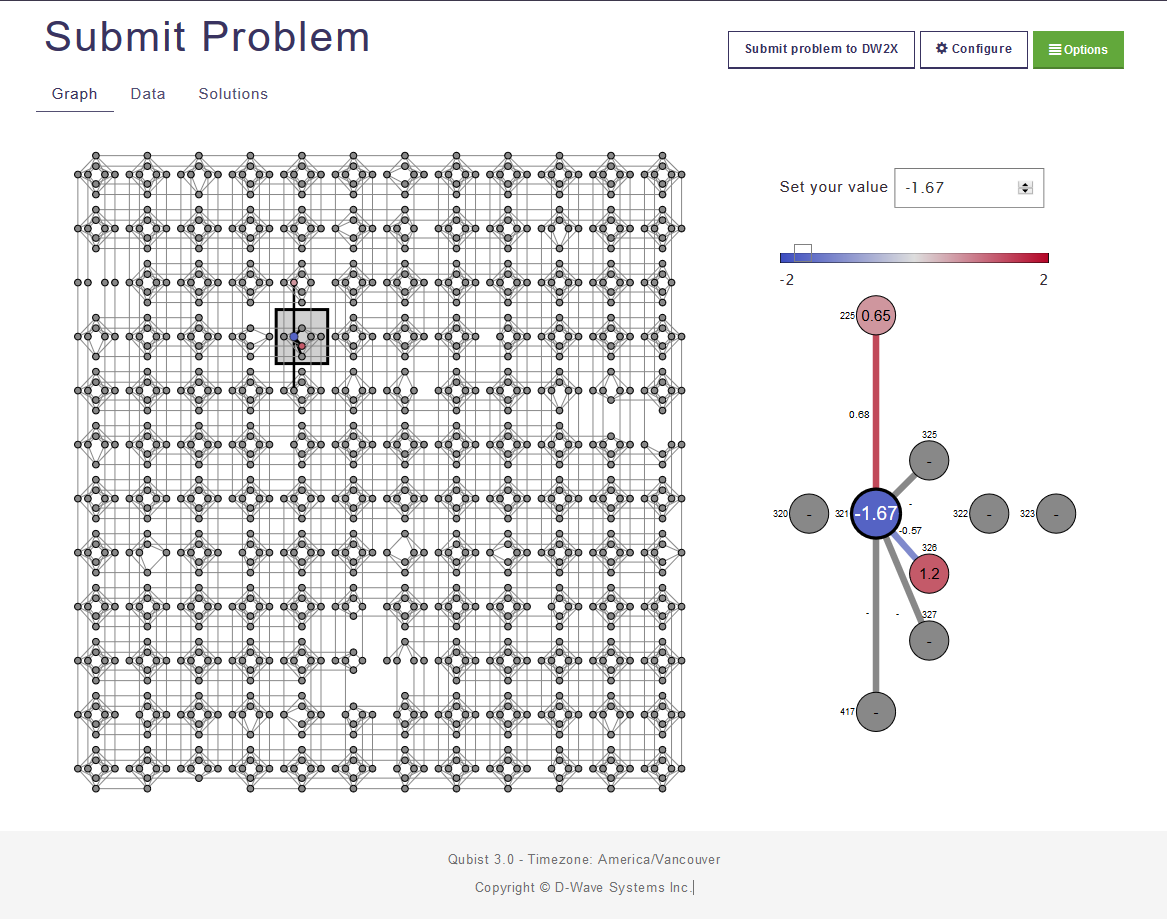}
\centering
\caption{Introducing a problem inside the D-Wave 2X}
\label{fig::dwavequbo}
\end{figure}

Please note that, in the previous paragraph, we have shown the solution that \verb|qbsolv| obtains for a certain instance of the problem. This is because \verb|qbsolv| has an auxiliary internal optimization solver based on the tabu search \cite{glover1986future}, which gives a solution for the subproblems and then unifies all the solutions into the complete one for all the variables. However, this classical method does not guarantee a global solution, and is of no help to us in the general case. We can decide if \verb|qbsolv| tries to solve the problem with its tabu search, of if we prefer to connect to a D-Wave machine. \\

Finally, our QUBO subproblem is solved with the D-Wave 2X via one of the possible inputs allowed. The first one is shown in Figure \ref{fig::dwavequbo}, while the second one is again a \verb|.qubo| file obtained internally with \verb|qbsolv|. As part of the project Joint Research Unit Repsol-ITMATI (code file: IN853A 2014/03), we had the opportunity to try the D-Wave 2X machine based on the University of Southern California. However, due to the amount of subproblems needed to solve a proper instance of the Frobenius problem or the Apéry set, it was impractical to do so with the amount of time given and the current size of the graph of the D-Wave 2X machine.\\

\section{A word on conclusions}

Regarding the algorithms for the Apéry set and the Frobenius number, there are two aspects that need to be improved prior to completing a study of its feasibility and performance. 

First, the current graph (i.e., Chimera graph) architecture of the available adiabatic quantum computers (i.e., the D-Wave machines) extremely obstruct the resolution of problems that have an almost full connectivity index between its variables, as in this case (the use of \verb|qbsolv| is just a temporary workaround, or it should be as so). 

And second, current adiabatic quantum computers do not guarantee global optimality, as opposed to the theoretical result deducted from the adiabatic theorem; in reality, the D-Wave just make a few runs of the process (instead of just one, as would be in the theoretical case) and, following a certain probability distribution, try to guarantee that the best of the solutions obtained by those runs is in fact the global solution to the problem. This solution, however, cannot be proven to be the global optimum (as global optimality is not known to be in the class NP), which makes useless our attempt to find those two combinatorial invariants via current adiabatic quantum computers. 

In the future, with more reliable quantum annealers, however, this solution may prove to be effective and faster, but for now it is just a theoretical method.

\bibliographystyle{siam}
\bibliography{refs}

\end{document}